\documentclass[8pt]{article}

\usepackage[utf8]{inputenc}
\usepackage[english]{babel}
\usepackage[margin=1.2cm]{geometry}
\usepackage{mathtools}
\usepackage{float}

\usepackage{pgf,tikz,pgfplots}
\pgfplotsset{compat=1.15}
\usepackage{mathrsfs}
\usetikzlibrary{arrows}
\usetikzlibrary{patterns}
\usepackage{xcolor}

\usepackage{xcolor}

\usepackage{amsmath,amsthm,amssymb}
\usepackage{hyperref}
\usepackage{enumerate}

\newtheorem{theorem}{Theorem}[section]
\newtheorem{definition}[theorem]{Definition}

\newtheorem{lemma}[theorem]{Lemma}
\newtheorem{proposition}[theorem]{Proposition}

\newtheorem{problem}[theorem]{Problem}

\theoremstyle{remark}
\newtheorem{remark}[theorem]{Remark}

\DeclareMathOperator{\Int}{Int}
\DeclareMathOperator{\conv}{conv}

\newcommand\R{\mathbb{R}}

\newcommand{\St}{{\mathcal St}}

        \newcommand{\eps}{\varepsilon}
        \newcommand{\HH}{{\mathcal H}}
        \renewcommand{\H}{\HH^1}
        
        \newcommand{\defeq}{:=}
        \newcommand{\forget}[1]{}

        \def\dist{\mathrm{dist}\,}

        \def\X{\mathrm{X}}
        \def\E{\mathrm{E}}
        \def\S{\mathcal S}

        \def\Int{\mathrm{Int}\,}

\begin{document}

\title{On minimizers of the maximal distance functional for a planar convex closed smooth curve}










\author{ D.D. Cherkashin$^\mathrm{a}$,~
 A.S. Gordeev$^{\mathrm{b,c}}$, G.A. Strukov$^{\mathrm{b,c}}$, Y.I. Teplitskaya$^\mathrm{a}$\\
{\small ~a. Chebyshev Laboratory, St. Petersburg State University, 14th Line V.O., 29B, Saint Petersburg 199178 Russia}\\
{\small ~b. St.~Petersburg Department of V.~A.~Steklov Institute of Mathematics of the Russian Academy of Sciences}\\
{\small ~c. The Euler International Mathematical Institute, St. Petersburg, Russia}
\date{}
}

\maketitle

\begin{abstract}
Fix a compact $M \subset \mathbb{R}^2$ and $r>0$.
A \textit{minimizer} of the maximal distance functional is a connected set $\Sigma$ of the minimal length, such that 
\[
\max_{y \in M} \dist(y,\Sigma) \leq r.
\]
The problem of finding maximal distance minimizers is connected to the Steiner tree problem.

In this paper we consider the case of a convex closed curve $M$, with the minimal radius of curvature greater than $r$ (it implies that $M$ is smooth). 
The first part is devoted to statements on structure of $\Sigma$: we show that the closure of an arbitrary connected component of $B_r(M) \cap \Sigma$ is a local Steiner tree which connects no more than five vertices.

In the second part we ``derive in the picture''. Assume that the left and right neighborhoods of $y \in M$ are contained in $r$-neighborhoods of different points $x_1$, $x_2 \in \Sigma$. We write conditions on the behavior of $\Sigma$ in the neighborhoods of $x_1$ and $x_2$ under the assumption by moving $y$ along $M$.
\end{abstract}

\section{Introduction}
For a given compact set $M \subset \mathbb{R}^2$ consider the maximal distance functional
\[
F_{M}(\Sigma)\defeq \max _{y\in M}\dist (y, \Sigma),
\]
where $\Sigma$ is a compact planar set, and $\dist(y, \Sigma)$ stands for the Euclidean distance between $y$ and $\Sigma$. Also $F_M(\emptyset) := \infty$. 

Consider the class of closed connected
sets $\Sigma \subset \R^2$ such that $F_M(\Sigma) \leq r$ for a given $r > 0$. We are interested in the properties of sets of minimal
length (one-dimensional Hausdorff measure) $\H(\Sigma)$ over the mentioned class. Further we call such sets \emph{minimizers}.

It is known that the set of minimizers is non-empty. It is also known that every minimizer $\Sigma$ of positive length satisfies $F_M(\Sigma)=r$. 
Also in this case the set of minimizers coincides with the set of solutions of the corresponding dual problem: to minimize $F_M$ among the class of closed connected sets $\Sigma \subset \mathbb{R}^2$ with the prescribed bound on the length $\H(\Sigma) \leq l$ (that is the reason of calling the desired set minimizers of maximal distance functional). General statements and details of the mentioned results can be found in~\cite{mir}.

Let $B_r (x)$ be the open ball of radius $r$ centered at a point $x$. Let $B_r(M)$ be the open $r$-neighborhood of $M$:
\[
B_r(M) \defeq \bigcup_{x\in M} B_r(x).
\]

\subsection{Properties of \texorpdfstring{$\Sigma$}{Sigma} for general \texorpdfstring{$M$}{M}}

In this subsection $M$ is an arbitrary planar compact.

Note that $\Sigma$ is bounded (and hence compact), since $\Sigma \subset \overline{B_r(\conv M)}$, where $\conv M$ stands for the convex hull of $M$. 
 
\begin{definition} 
A point $x \in \Sigma$ is called energetic, if for all $\rho > 0$ the set $\Sigma \setminus B_\rho(x)$ does not cover $M$ i.e.
\[
\dist(M, \Sigma \setminus B_\rho(x)) > r.
\]
Denote the set of energetic points by $G_\Sigma$.
\end{definition}

Every minimizer $\Sigma$ can be split into three disjoint subsets:
\[
        \Sigma=\E_{\Sigma}\sqcup\X_{\Sigma}\sqcup\S_{\Sigma},
\]
where $X_\Sigma \subset G_\Sigma$ is the set of \textit{isolated energetic} points (i.e. every $x \in X_\Sigma$ is energetic and there is a $\rho > 0$ such that $B_\rho(x) \cap G_\Sigma = \{x\}$), $E_\Sigma := G_\Sigma \setminus X_\Sigma$ is the set of \textit{non-isolated energetic}
points and $S_\Sigma := \Sigma \ G_\Sigma$ is the set of non-energetic points also called the \textit{Steiner part} of $\Sigma$.

The following basic properties of minimizers has been proved in~\cite{mir} (for planar $M$) and in~\cite{PaoSte04max} (for $M \subset \mathbb{R}^n$):
        
\begin{itemize}
    \item[(a)] minimizers contain no cycles (homeomorphic images of circumference). 
    \label{a}
    \item[(b)] For every energetic $x \in G_\Sigma$ there is a point $y \in M$, such that $|x-y|=r$ and $B_{r}(y)\cap \Sigma=\emptyset$. 
    Further we call $y$ \textit{corresponding} to $x$ and denote by $y(x)$. Note that a corresponding point may be not unique.
    \label{b}
    \item[(c)] For every non-energetic $x \in S_\Sigma$ there is an $\varepsilon > 0$, such that $\Sigma \cap B_{\varepsilon}(x)$ is either a segment or a \textit{regular tripod}, i.e. the union of three line segments with an endpoint in $x$ and relative angles of $2\pi/3$.
    \label{Sbasis}
    \end{itemize}

\begin{theorem}[Teplitskaya,~\cite{teplitskaya2018regularity,teplitskaya2019regularity}]
Let $\Sigma$ be a maximal distance minimizer for a compact set $M \subset \mathbb{R}^2$, $r > 0$. We say that the ray $ (ax] $ is a \textit{tangent ray} of the set $ \Sigma $ at the point $ x\in \Sigma $ if there exists non  stabilized sequence of points $ x_k \in \Sigma $ such that $ x_k \rightarrow x $ and $ \angle x_kxa \rightarrow 0 $. Then \begin{itemize}
\item[(i)] $\Sigma$ is a union of a finite number of injective images of the segment $[0,1]$;
\item[(ii)] the angle between each pair of tangent rays at every point of $\Sigma$ is greater or equal to $2\pi/3$;
\item[(iii)] the number of tangent rays at every point of $\Sigma$ is not greater than $3$. If it is equal to $3$, then there exists such a neighbourhood of $x$ that the arcs in it coincide with line segments and the pairwise angles between them are equal to $2\pi/3$. 
\end{itemize}
\end{theorem}

\subsection{The class of \texorpdfstring{$M$}{M}, considered in the paper}
\label{setup}

Fix a positive real $r$ and a closed convex curve $M$ with the minimal radius of curvature $R > r$ (this implies $C^{1,1}$-smoothness of $M$).
Introduce the notation: $N := \conv (M)$; let $M_r$ be the inner part of the boundary of $B_r(M)$, and finally put $N_r = \conv(M_r)$.
Note that $M_r$ also is a closed convex curve $M$ with the minimal radius of curvature $R - r$.
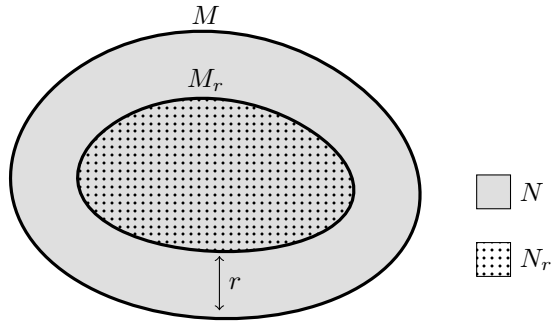
\begin{figure}[h]
    \centering
    \begin{tikzpicture}[scale=0.9]
    \fill[lightgray!50] plot [smooth cycle, tension=1] coordinates 
        {(0,0) (0,3) (4,3) (4.5, 0)};
    \draw[very thick] plot [smooth cycle, tension=1] coordinates 
        {(0,0) (0,3) (4,3) (4.5, 0)};
    \draw[very thick, pattern=dots] plot [smooth cycle, tension=1] coordinates 
        {(0.7,0.7) (0.7,2.3) (3.3,2.3) (3.8, 0.7)};
    \draw[<->] (2.2,-0.5) -- (2.2,0.3);
        
    \node[above] at (2, 3.6) {$M$};
    \node[above] at (2, 2.6) {$M_r$};
    \node[right] at (2.2, -0.1) {$r$};
    
    \draw[fill = lightgray!50] (6,1) rectangle (6.5,1.5);
    \draw[pattern=dots] (6,0) rectangle (6.5,0.5);
    \node[right] at (6.5, 1.25) {$N$};
    \node[right] at (6.5, 0.25) {$N_r$};
\end{tikzpicture}

    \caption{Definition of $N$, $M_r$ and $N_r$}
    \label{MMrNNr}
\end{figure}

Further $\Sigma$ denotes an arbitrary minimizer for $M$.

\subsection{The problem for particular \texorpdfstring{$M$}{M}}

Finding the set of minimizers for almost every particular $M$ is quite difficult. There are the following results.

\begin{figure}[h]
	\begin{center}
		\definecolor{yqqqyq}{rgb}{0.,0.,1.}
		\definecolor{xdxdff}{rgb}{0.,0.,0.}
		\definecolor{ffqqqq}{rgb}{0.,0.,1.}
    \definecolor{ttzzqq}{rgb}{0.,0.,0.}
    \begin{tikzpicture}[line cap=round,line join=round,>=triangle 45,x=1.0cm,y=1.0cm]
    \clip(-5.754809259201689,-2.15198488708635626) rectangle (0.492320066152345,2.2268250892462824);
    \draw [very thick, color=ttzzqq] (-3.,2.)-- (-2.,2.);
    \draw [very thick, color=ttzzqq] (-2.,-2.)-- (-3.,-2.);
    \draw [very thick, color=ffqqqq] (-3.,-1.)-- (-2.,-1.);
    \draw [dotted] (-3.,1.)-- (-2.,1.);
    \draw [shift={(-3.,0.)},very thick, color=ttzzqq]  plot[domain=1.5707963267948966:4.71238898038469,variable=\t]({1.*2.*cos(\t r)+0.*2.*sin(\t r)},{0.*2.*cos(\t r)+1.*2.*sin(\t r)});
    \draw [shift={(-3.,0.)},dotted]  plot[domain=1.5707963267948966:4.71238898038469,variable=\t]({1.*1.*cos(\t r)+0.*1.*sin(\t r)},{0.*1.*cos(\t r)+1.*1.*sin(\t r)});
    \draw [shift={(-2.,0.)},dotted]  plot[domain=-1.5707963267948966:1.5707963267948966,variable=\t]({1.*1.*cos(\t r)+0.*1.*sin(\t r)},{0.*1.*cos(\t r)+1.*1.*sin(\t r)});
    \draw [shift={(-2.,0.)},very thick,color=ttzzqq]  plot[domain=-1.5707963267948966:1.5707963267948966,variable=\t]({1.*2.*cos(\t r)+0.*2.*sin(\t r)},{0.*2.*cos(\t r)+1.*2.*sin(\t r)});
    \draw [dash pattern=on 2pt off 2pt] (-2.5008100281931047,2.) circle (1.cm);
    \draw [very thick,color=yqqqyq] (-3.7309420990521054,0.6824394829091469)-- (-3.1832495111022516,1.2690579009478955);
    \draw [very thick,color=yqqqyq] (-1.8178907146013357,1.2695061868000874)-- (-1.2695061868000874,0.6829193135917684);
    \draw [shift={(-2.,0.)},very thick,color=ffqqqq]  plot[domain=-1.5707963267948966:0.7517515639553677,variable=\t]({1.*1.*cos(\t r)+0.*1.*sin(\t r)},{0.*1.*cos(\t r)+1.*1.*sin(\t r)});
    \draw [shift={(-3.,0.)},very thick,color=ffqqqq]  plot[domain=2.390497746093128:4.771687465439039,variable=\t]({1.*1.*cos(\t r)+0.*1.*sin(\t r)},{0.*1.*cos(\t r)+1.*1.*sin(\t r)});
    \begin{scriptsize}
    \draw[color=ttzzqq] (-4.8089822121844977,0.2025724503290166) node {$M$};
    \draw[] (-3.763965847825129,-0.1818662773850787) node {$\Sigma$};
    \draw[color=black] (-2.059620164509451,0.8028001329511188) node {$M_r$};
    \end{scriptsize}
    \end{tikzpicture}
    \caption{A horseshow}
    \label{horseshoe}
	\end{center}
    \end{figure}
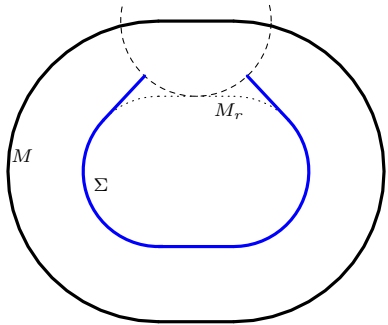

\begin{theorem}[Cherkashin -- Teplitskaya, 2018~\cite{cherkashin2018horseshoe}]
Let $r$ be a positive real, $M$ be a convex closed curve with the radius of curvature at least $5r$ at every point, $\Sigma$ be an arbitrary minimizer for $M$. Then $\Sigma$ is a union of an arc of $M_r$ and two segments, that are tangent to $M_r$ at the ends of the arc (so-called horseshoe, see Fig.~\ref{horseshoe}). 
In the case when $M$ is a circumference with radius $R$, the claim is true for $R > 4.98r$.
\label{horseshoeT}
\end{theorem}

We prepare the paper with the following theorem.

\begin{theorem}
Let $M = A_1A_2A_3A_4$ be a rectangle, $0 < r < r_0(M)$. Then a maximal distance minimizer has the following topology, depicted in the left part of Fig.~\ref{rectangle}. The middle part of the picture contains enlarged fragment of the minimizer near $A_1$; the labeled angles are equal to $\frac{2\pi}{3}$. The rightmost part contains much more enlarged fragment of minimizer near $A_1$.

A minimizer consists of 21 segments; an approximation of the length of a minimizer is $Per - 8.473981r$, where $Per$ is the perimeter of the rectangle.
\label{rectangleT}
\end{theorem}

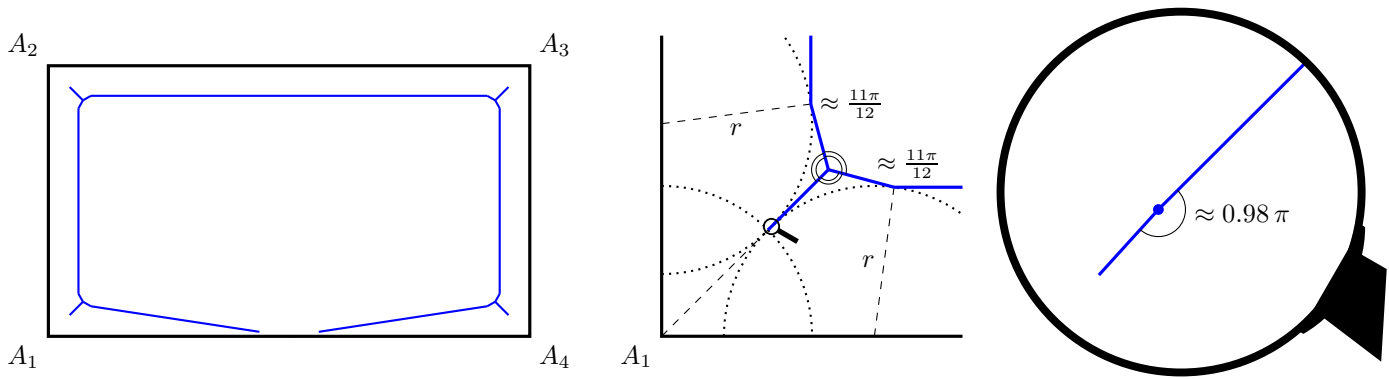
\begin{figure}[H]
    \centering
    \hfill
    \begin{tikzpicture}[scale=0.4]
        
        \def\w{8}
        \def\h{4.5}
        
        \def\t{.85}
        \def\c{}
        
        \foreach\x in {-1,1} {
            \foreach\y in {-1,1} {
                \draw[blue, thick]  ({\x*(-\w+1 - .292893218813455)},
                        {\y*(-\h+1 - .2928932188134548}) -- 
                       ({\x*(-\w+2 - \t)},
                        {\y*(-\h+2 - \t)});
                \draw[blue, thick]  ({\x*(-\w+2 - \t)},
                        {\y*(-\h+2 - \t)}) -- 
                       ({\x*(-\w+1)}, 
                        {\y*(-\h+2 -.5857864376269108)});
                \draw[blue, thick]  ({\x*(-\w+2 - \t)},
                        {\y*(-\h+2 - \t)}) -- 
                       ({\x*(-\w+2 - .5857864376269095)},
                        {\y*(-\h+1)});
            }
        }
        \draw[blue, thick]  (-\w+2 - .5857864376269095,-\h+1) -- 
               ( 0,-\h);
        \draw[blue, thick]  ( 0,-\h) -- 
               ( \w-2 + .5857864376269095,-\h+1);
        \draw[blue, thick]  (-\w+1, -\h+2 - .5857864376269108) -- 
               (-\w+1,  \h-2 + .5857864376269108);
        \draw[blue, thick]  (-\w+2 - .5857864376269095, \h-1) -- 
               ( \w-2 + .5857864376269095, \h-1);
        \draw[blue, thick]  ( \w-1, -\h+2 - .5857864376269108) -- 
               ( \w-1,  \h-2 + .5857864376269108);
               
        \fill[white] (1, -\h) arc (0:180:1);
        
        \draw [very thick] (-\w,-\h) rectangle (\w,\h);
        
        \draw (-\w, -\h) node[below left]{$A_1$};
        \draw (-\w, \h) node[above left]{$A_2$};
        \draw (\w, \h) node[above right]{$A_3$};
        \draw (\w, -\h) node[below right]{$A_4$};
        
    \end{tikzpicture}
\hfill
    \begin{tikzpicture}[scale=2]
        
        \coordinate (V) at (1.10837, 1.10837);
        \coordinate (Q) at (.7071, .7071);
        \coordinate (C) at (.73, .73);
        
        \draw[dashed] (0,0) -- (Q);
        \draw[thick, dotted] (0, 1) arc (90:0:1);
        \draw[thick, dotted] (.414, 0) arc (180:55:1);
        \draw[dashed] (1.414, 0) -- (1.544739754593147, 0.9914448613738104) node[pos=0.5, left]{$r$};
        \draw[thick, dotted] (0, .414) arc (-90:35:1);
        \draw[dashed] (0, 1.414) -- (0.9914448613738104,1.544739754593147) node[pos=0.5,below]{$r$};
        
        \draw [very thick] (2,0) -- (0,0) -- (0,2);
        \draw (0, 0) node[below left]{$A_1$};
        \draw [very thick, blue] (0.9914448613738104,2)-- (0.9914448613738104,1.544739754593147);
        \draw [very thick, blue] (V)-- (0.9914448613738104,1.544739754593147);
        \draw [very thick, blue] (V)-- (1.5447397545931463,0.9914448613738106);
        \draw [very thick, blue] (1.5447397545931463,0.9914448613738106)-- (2,0.9914448613738102);
        \draw [very thick, blue] (Q)-- (V);    

        \def\rr{0.12}
        \draw[shift={(V)}] 
            (0.966*\rr ,-0.2588*\rr) arc(-15:105:\rr);
        \def\rr{0.09}
        \draw[shift={(V)}] 
            (0.966*\rr ,-0.2588*\rr) arc(-15:105:\rr);
        \def\rr{0.096}
        \draw[shift={(V)}] 
            (0.966*\rr ,-0.2588*\rr) arc(-15:-135:\rr);
        \def\rr{0.072}
        \draw[shift={(V)}] 
            (0.966*\rr ,-0.2588*\rr) arc(-15:-135:\rr);
        \def\rr{0.112}
        \draw[shift={(V)}] 
            (-0.2588*\rr ,0.966*\rr) arc(105:225:\rr);
        \def\rr{0.082}
        \draw[shift={(V)}] 
            (-0.2588*\rr ,0.966*\rr) arc(105:225:\rr);
        
        \node[right] at (0.9914448613738104,1.544739754593147) {\small $\approx \frac{11\pi}{12}$};
        \node[above] at (1.644739754593147, 0.9914448613738104) {\small $\approx \frac{11\pi}{12}$};
        
        \def\ll{0.051}
        \def\llll{0.2}
        \draw[line width = .7pt] (C) circle (\ll);
        \draw[shift={(C)}, line width = 2pt] (\ll * .866, \ll * -.5) -- (\llll * 0.866, \llll * -.5);
    \end{tikzpicture}
    \hfill
    \begin{tikzpicture}[scale=48]
        \coordinate (V) at (.7645532062, .76593);
        \coordinate (Q2) at (.723714, .725155);
        \coordinate (Q1) at (.707224, .706989);
        \coordinate (C) at (.73, .73);

        \draw[very thick, blue] (V) -- (Q2) -- (Q1);
        \fill[blue] (Q2) circle (0.0015);
        
        \draw[shift={(Q2)}, rotate=45] (0.0075, 0) arc (0 : -177 : 0.0075) node[right, pos=0.3]{$\approx 0.98\, \pi$};
        
        \fill[black, shift={(C)}, rotate=60] (-0.01, -0.048) -- (-0.013, -0.0715) -- (0.01, -0.06) -- (0.01, -0.048) -- cycle;
        \draw[black, line width=3] (C) circle (0.05);
        
        \fill[black, shift={(C)}, rotate=60] (0.0171, -0.046) arc (-70 : -110 : 0.05) -- (-0.0171, -0.049) arc (-120 : -60 : 0.0342) -- cycle;
    \end{tikzpicture}
    \hfill\ 
\caption{The minimizer for rectangle $M$ and $r < r_0(M)$.}

    \label{rectangle}
\end{figure}

\paragraph{Structure of the paper.} Section~\ref{stp} contains an introduction to the Steiner problem. Section~\ref{struct} is devoted to structural properties of $\Sigma$. In Section~\ref{diff} we ``derive in the picture''. Finally, Section~\ref{rest} contains applications of our methods and open questions.

\section{Steiner tree problem} \label{stp}

Consider a finite set of points $C \defeq \{ A_1, \dots, A_n \} \subset \R^2$. A \textit{Steiner tree} is a connected set $S\subset \R^2$, which contains $C$ and has minimal possible length. 
It is known that such $S$ always exists (but is not necessarily unique) and that it is the union of a finite set of segments. Thus, $S$ can be represented as a plane graph, such that its set of vertices contains $C$, and all its edges are straight line segments. This graph is connected and does not contain cycles, i.e. is a tree, which explains the naming of $S$. It is known that the maximum degree of this graph is no greater than $3$. Moreover, only vertices $A_i$ can have degree $1$ or $2$, all the other vertices have degree $3$ and are called \textit{Steiner points}. There are no more than $n-2$ Steiner points. The angle between any two adjacent edges is at least $2\pi/3$. That means that for a Steiner point the angle between any two edges incident to it is exactly $2\pi/3$. $S$ is called a \textit{full Steiner tree}, if the degree of each $A_i$ is $1$, or, equivalently, if the number of Steiner points is $n-2$.
A \textit{(full) Steiner forest} is a set, each connected component of which is a (full) Steiner tree.
Proof of the listed properties of Steiner trees and additional information on them can be found in the book~\cite{hwang1992steiner} and in the article~\cite{gilbert1968steiner}.


We define a \textit{local Steiner tree} as a connected compact acyclic set $S$, which contains $C$, and such that for any $x \in S \setminus C$ there is a neighborhood $U\ni x$ such that $S \cap U$ coincides with the Steiner tree on the set of points $S \cap \partial U$.  A local Steiner tree retains the following properties of a Steiner tree: it is the union of a finite set of segments; the angle between any two adjacent segments is at least $2\pi/3$.
We are going to use the following fact: a connected closed subset of a local Steiner tree is itself a local Steiner tree.

For a given tree $T$ we denote the set of its vertices of degree 1 or 2 as $\partial T$.

\begin{definition}
Define a ``wind rose'' as a set of six rays starting at the origin point with angle $\pi/3$ between any two adjacent rays; each ray is given a weight (a real number), which satisfy the following property: the weight of a ray is the sum of weights of two rays adjacent to it. (It follows, in particular, that the sum of the weights of two opposite rays (the ones forming a line) is zero.) 
\label{roza_v}
\end{definition}

By \textit{full Steiner pseudo-network} let us call a connected set $S$ which contains $C$, if for any wind rose $\mathcal R$ such that
\begin{itemize}
    \item [(i)] $S$ consists of finite number of segments which are parallel to $\mathcal R$
\end{itemize}
the following holds:
\begin{itemize}
    \item [(ii)] for any $x \in S \setminus C$ and small enough $\varepsilon > 0$, sum of weights of rays of $\mathcal R$ which are parallel to rays of the form $[xy)$, $y \in \partial B_\varepsilon (x) \cap S$, is zero.
\end{itemize}

It is clear that full (local) Steiner tree is a full Steiner pseudo-network.

For a given pseudo-network $T$ let us denote by $\partial T$ set of vertices of degree 1.

\begin{remark}
Suppose that $T$ is a full Steiner pseudo-network, and $\mathcal R$ is an arbitrary wind rose satisfying~(i). Let us assign to an each vertex $x \in \partial T$ a weight of a ray of $\mathcal R$, which is parallel to a directed segment of $T$ entering $x$ (such segment is unique by definition of $\partial T$). Then sum of assigned numbers over all $x\in \partial T$ is zero.
\label{main_rose_property}
\end{remark}


\begin{lemma}
Let $T$ be a full Steiner pseudo-network, $l$ be a line such that $T \not \subset l$. Then
\[  
\sharp(\partial T \cap l)\leq 2\sharp(\partial T \setminus l).
\]
\label{St_l}
\end{lemma}

\begin{proof}[Proof of Lemma~\ref{St_l}:]

Let $l^{+}$, $l^{-}$ be the two open half-planes bounded by $l$. Note that it is sufficient to prove the inequality for a closure of an arbitrary component of  $\overline{T\cap l^{+}}$ and $\overline{T\cap l^{-}}$,
denote such closure as $S$.

Consider a wind rose with the origin in the same open half-plane as $S$, such that the rays with positive weights are exactly the ones intersecting $l$: such wind rose exists, because $l$ intersects either $2$ or $3$ consecutive (in the counter-clockwise order) rays. In the former case we give these rays weights $1, 1$, in the latter case~--- $1, 2, 1$. The remaining rays will have weights $0, -1, -1, 0$ or $-1, -2, -1$ (the weights are listed in counter-clockwise order in each case). We assign weights to all leaf vertices in the way described in Remark~\ref{main_rose_property}. Then the sum of weights over the leaf vertices lying on $l$ is at least $\sharp(S \cap l)$. Since, according to Remark~\ref{main_rose_property}, the sum over all leaf vertices should be zero, there are at least $\frac{\sharp(S \cap l)}{2}$ leaf vertices not lying on $l$.

\end{proof}

\begin{remark}
\label{corollarystl}
Let $T$ be a full Steiner pseudo-network fully lying on one side of line $l$, such that equality in Lemma~\ref{St_l} is achieved. Then all leaf vertices in $\partial T \setminus l$ have weight $-2$, therefore all segments of $T$ incident to vertices from $\partial T \setminus l$ are pairwise collinear.
\end{remark}

\section{Structural properties of minimizers}
\label{struct}
Recall that we work in the setting from Subsection~\ref{setup}. 

Note that $\Sigma \subset N$ ($N$ is convex, so one can project the part of $\Sigma$ belonging to $\mathbb{R}^2 \setminus N$ on $N$ and length of
$\Sigma$ will strictly decrease).

Consider the closure of an arbitrary connected component of $\Sigma \setminus N_r$; denote it by $S$. Points from $S\cap M_r$ are called \textit{entering points}. 
Connectedness of $S$ implies that $\overline{B_r(S)} \cap M$ is a closed arc; denote it by $q(S)$.

The following lemma is proved in~\cite{cherkashin2018horseshoe} (the proof of these statements does not use the additional requirement $R > 5r$, which is inherited from the main theorem of the paper~\cite{cherkashin2018horseshoe}).

\begin{lemma}
\label{3points}
Let $S$ be the closure of a connected component of $\Sigma \setminus N_r$. 
Then 
\begin{itemize}
    \item [(i)] $S$ is a local Steiner tree connecting the set of entering points of $S$ and energetic points of $S$;
    \item [(ii)] $S$ contains one or two energetic points.
    \item [(iii)] Suppose that $S$ contains 2 energetic points $x_1$ and $x_2$.
Then
\begin{itemize}
    \item [(i)] there are unique points $y(x_1)$ and $y(x_2)$; \label{2pointsunique}
    \item [(ii)] if $x_i$ has degree 1 (i.e. $x_i$ is the end of a line segment $[z_ix_i] \subset \Sigma$), then $z_i$, $x_i$ and $y(x_i)$ are collinear; \label{stepen1}  
    \item [(iii)] if $x_i$ has degree 2 (i.e. $x_i$ is the end of a line segments $[z_i^1x_i], [x_iz_i^2] \subset \Sigma$), then ray $[y(x_i)x_i)$ contains the bisector of $z_i^1x_iz_i^2$. \label{stepen2} 
\end{itemize}

\end{itemize}
\end{lemma}

In this section we prove the following statement.

\begin{proposition}
Let $S$ be the closure of a connected component of $\Sigma \setminus N_r$. 
Then 
\begin{itemize}
\item [(i)] the convex hull of $S$ is a line segment, a triangle or a quadrangle; the vertices of convex hull are only energetic or entering points of $S$, the latter no more than 2;
\item [(ii)] $S$ has at most 3 entering points.
\end{itemize}
\end{proposition}
\paragraph{Proof of~(i).} Since $S$ is a local Steiner tree for its entering and energetic points, every other point $x$ is a convex combination of points from a neighborhood of $x$. So it is enough to show that all entering points except at most two lie in the interior of $\conv(S)$.

Suppose the contrary and consider maximal (by inclusion) arc $A \subset M_r$ ending by entering points of $S$ (further we call them \textit{extreme}), that $A \subset \overline{B_r(q(S))}$. Consider an arbitrary entering point $x$ lying in the interior of the arc and set the tangent line to $M_r$ at $x$. Since $N_r$ is convex, connected component $S$ contains points in the both sides of the tangent line, say $t_1$ and $t_2$. Then $x$ is a convex combination of $t_1$ and $t_2$; which is a contradiction.  \qed

\begin{figure}
\centering
    \hfill
    \begin{tikzpicture}[line cap=round,line join=round,>=triangle 45,x=1cm,y=1cm, scale =12]
    
    \def\area{(-6.4, 1.4) rectangle (-5.84, 1.7)};
    \clip \area;
    
    \draw [dashed,domain=-7.000492781939422:-5.045681011148128] plot(\x,{(--3.5920094869218673-0*\x)/2.204854695810054});
    
    \node[black, above] at (-5.9,1.6291366019483557) {$M_r$};
        
    \draw [,dotted] (-6.78334510611128,0.6142852583161696) circle (1.014851343632186cm);
    \draw [dotted] (-5.466807795436834,0.6142852583161696) circle (1.014851343632186cm);
    
    \draw [very thick, blue] (-6.09955343275884,1.5500731447679825)-- (-6.184596413792117,1.4336894645038996);
    \draw[] (-6.184596413792117,1.4336894645038996) -- (-6.78334510611128,0.6142852583161696);
    \draw [very thick, blue] (-6.123454498898773,1.6041517478095404)-- (-6.09955343275884,1.5500731447679825);
    \draw [very thick, blue] (-6.355098527042809,1.629136601948356)-- (-6.123454498898773,1.6041517478095404);
    \draw [very thick, blue] (-6.105197759879817,1.6291366019483557)-- (-6.123454498898773,1.6041517478095404);
    
    \draw [] (-5.905972590576058,1.529193741380122)-- (-5.466807795436834,0.6142852583161696);
    \draw [very thick, blue] (-5.905972590576058,1.529193741380122)-- (-6.09955343275884,1.5500731447679825);
    \draw [very thick, blue] (-5.857448245737062,1.6051829270264233)-- (-5.905972590576058,1.529193741380122);
    
    \def\nodesize{0.006}
    
    \fill[blue] (-6.09955343275884,1.5500731447679825) circle (\nodesize)
        node[black, left]{$v_1$};
    \fill[blue] (-6.184596413792117,1.4336894645038996) circle (\nodesize)
        node[black, right]{$x_1$};
    \fill[blue] (-5.638657348174386,0.6142852583161696) circle (\nodesize)
        node[black, left]{$b$};
    \fill[blue] (-6.355098527042809,1.629136601948356) circle (\nodesize)
        node[black, above]{$U_1$};
    \fill[blue] (-6.105197759879817,1.6291366019483557) circle (\nodesize)
        node[black, above]{$U_2$};
    \fill[white] (-5.935,1.51) circle (0.015);
    \fill[blue] (-5.905972590576058,1.529193741380122) circle (\nodesize)
        node[black, below left] {$v_2 = x_2$};
    \fill[blue] (-6.123454498898773,1.6041517478095404) circle (\nodesize);
        \node[black, right] at(-6.12, 1.595){$V$};
    \end{tikzpicture}
    \hfill
    \begin{tikzpicture}[line cap=round,line join=round,>=triangle 45,x=1cm,y=1cm,scale=1.45]
        
        \def\area{(-6.7,-1.2) rectangle (-3.3,1.2)};
        \clip \area;
        
        \fill[gray!20]
        (-6.144235945986252,-0.4940859356732673) --
        (-4.987071803136054,0.7986098324835481) --
        (-4.140753395373325,0.6212882232099375) --
        (-3.7597158072642687,-0.8068658531644282) --
        (-5.006891543175035,-0.98100560400365) -- cycle;
        \node at (-4.9,-.3) {\large $P$};
        
        \draw [very thick, %
            rotate around={%
                -2.659749941701052:(-4.447051207241462,2.388595008642463)%
            },] %
            (-4.447051207241462,2.388595008642463) ellipse (3.8449094705207605cm and 3.237364037629703cm);
        \draw [,dotted] (-6.144235945986252,-0.4940859356732673) circle (1.4781115348364897cm);
        \draw [,dotted] (-3.759715807264265,-0.8068658531644273) circle (1.4781115348364897cm);
        \draw [] (-4.140753395373325,0.6212882232099375)-- (-3.759715807264265,-0.8068658531644273);
        \draw [blue, very thick] (-5.067152322607865,1.042656927956234)-- (-4.987071803136054,0.7986098324835481);
        \draw [blue, very thick] (-3.6869909015675746,1.025732914479689)-- (-4.140753395373325,0.6212882232099375);
        \draw [blue, very thick] (-4.987071803136054,0.7986098324835481)-- (-4.140753395373325,0.6212882232099375);
        \draw [blue, very thick] (-4.987071803136054,0.7986098324835481)-- (-5.1583825277993,0.607234520536362);
        \draw [] (-5.1583825277993,0.607234520536362)-- (-6.144235945986252,-0.4940859356732673);
        \draw [] (-5.006891543175035,-0.98100560400365)-- (-2.546267310088977,-0.6374353448006458);
        \draw [] (-6.753323099257971,-0.23332366004039778)-- (-5.006891543175035,-0.98100560400365);

        \def\nodesize{0.05}
        \fill[] (-6.144235945986252,-0.4940859356732673) circle (\nodesize);
            \node[below] at (-6.3, -0.5) {$y(x_1)$};
        \fill[blue] (-5.1583825277993,0.607234520536362) circle (\nodesize);
            \node[left] at (-5.18, .58) {$x_1$};
        \fill[] (-3.759715807264265,-0.8068658531644273)  circle (\nodesize)
            node[below]{$y(x_2)$};
        \fill[blue] (-4.140753395373325,0.6212882232099375) circle (\nodesize);
            \node[right] at (-4.14, .5) {$x_2$};
        
    \end{tikzpicture}
    \hfill\phantom{}
    \caption{Illustration to the proof of Lemma~\ref{3points}~(ii)}
    \label{3pointsFig}
\end{figure}
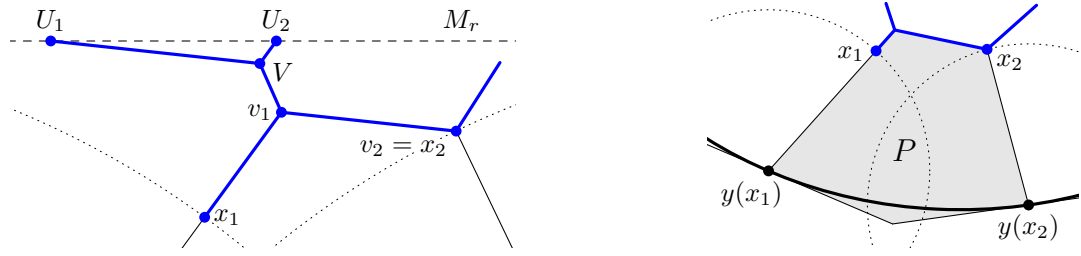

\paragraph{Proof of~(ii).} 
Denote extreme (defined analogously to the previous proof) entering points by $Y_1$ and $Y_2$.
For every other entering point $Y \in S$ denote by $R = R(Y)$ a continuation of a segment of $S$ which contains $Y$ beyond the point $Y$. Let us show that $R$ intersects with a line $Y_1Y_2$.
Note that $Y_1$, $Y$, $Y_2$ are contained in arc $M_r \cap \overline{B_r(q(S))} =: Q(S)$.

Suppose the contrary, that is $R$ either meets again with $M_r$ at $U \in Q(S)$ or is tangent to the $M_r$. 
Let us show that on the arc $YU \subset Q(S)$ there will be an entering point $Y'$ that belongs to the closure of another connected component $S'$.
If $R$ is tangent to $M_r$ in $Y$, then $Y$ lies in a closure of another component, therefore we can put $Y' = Y$.
Otherwise $Y$ lies in a closure of a component of $\Sigma \cap \Int(N_r)$, denote this component as $T$.
Since $\Sigma \cap \Int(N_r)$ lies in the Steiner part of $\Sigma$, a closure of any component of $\Sigma \cap \Int(N_r)$ is a full local Steiner tree. Then $\partial T \setminus \{Y\}$ contains a point in each of closed half-planes divided by $YU$; since $\partial T \subset M_r$, there is a vertex $Y' \in \partial T$ on the arc $YU$, since $\Sigma$ is acyclic, $Y'$ is not contained in $S$.

But then $q(S') \subset Q(S)$ or $S \cap S' \neq \emptyset$. It follows from the first option that $S'$ contains no energetic points; the second option is impossible by definition.

For each entering point $Y$ denote by $I(Y)$ an intersection of $R(Y)$ with $Y_1Y_2$. By $\St(S)$ we denote the union of $S$ and all segments $[YI(Y)]$.


Let us consider two cases.

\begin{itemize}
\item [(a)] 
Let $S$ have one energetic point $x$. 
\begin{itemize}
\item If $x$ has degree $1$, then $\St(S)$ is a full Steiner pseudo-network, then by application of Lemma~\ref{St_l} to $\St(S)$ and $Y_1Y_2$, we have that $\St(S)$ intersects with $Y_1Y_2$ at most two times, thus $S$ has at most two entering points.

\item If $x$ has degree $2$ then $\St(S)$, being cut in the point $x$, falls apart into two full networks $\St_1$ and $\St_2$. 
Let one of them have at least two entering points (say, $\St_1$).
By Lemma~\ref{St_l} for $\St_1$ and $Y_1Y_2$, it can be only a tripod; denote by $V_1$ the branching point of the tripod and by $U_1^1$ and $U_1^2$ the points of intersection with $Y_1Y_2$. 
Let $\St_2$ also be a tripod with branching point $V_2$ and 
with points  $U_2^1$ and $U_2^2$ on the line $Y_1Y_2$; without loss of generality we can assume that points $U_1^2$ and $U_2^1$ are lying between points $U_1^1$ and $U_2^2$.
Then the sum of angles of pentagon $U_1^1V_1xV_2U_2^2$ is at least $10\pi/3$, because
$\angle U_1^1V_1x = \angle xV_2U_2^2 = 2\pi/3$ (these angles are external for the pentagon and corresponding inside angles are equal to $4\pi/3$), $\angle V_1xV_2 \geq 2\pi/3$. That is a contradiction.

Summing up, no subtree can contain three entering points, and subtrees can not both contain two entering points simultaneously, which finishes this case.
\end{itemize}

\item[(b)] 
Let $S$ have two energetic points $x_1$ and $x_2$. Consider a polygon $P$ (see the right-hand side of Figure~\ref{3pointsFig}), that bounded by $S$, by segments $[x_1y(x_1)]$, $[x_2y(x_2)]$ and by tangents to $M$ in points $y(x_1)$ and $y(x_2)$ (by Lemma~\ref{2pointsunique} points $y(x_1)$ and $y(x_2)$ are unique). Note that $P$ is convex and its angles at vertices from $S$ are at most $2\pi/3$. Since $B_r(y(x_1)) \cap \Sigma = B_r(y(x_2)) \cap \Sigma = \emptyset$, angles at $y(x_1)$ and $y(x_2)$ are at most $\pi/2$; by Lemma~\ref{3points}(iii) the line $y(x_1)x_1$ contains a side of $P$, thus angle at $y(x_1)$ is less than $\pi/2$.  If $P$ has at least 3 vertices from $S$, then the sum of external angles of $P$ is strictly greater than $3\pi/3+2\pi/2 = 2\pi$, what is impossible. Therefore, $P$ contains no more than two vertices from $S$.

Let $x_i$ have degree 1. Then, if it is connected with $x_{3-i}$ by a segment of $\Sigma$, segment $[x_1x_2]$ can be removed from $\St(S)$ and by Lemma~\ref{St_l} remaining full Steiner network has no more than two points of intersection with $Y_1Y_2$, so $S$ has no more than two entering points. In the other case $x_i$ is connected by a segment of $\Sigma$ with a branching point $V_i$. Let $v_i = V_i$, if $x_i$ has degree 1, and $v_i = x_i$, if $x_i$ has degree 2 (see the left-hand side of Figure.~\ref{3pointsFig}). Then $v_1$ and $v_2$ are vertices of $P$. Since $P$ contains no more than 2 vertices from $S$, it turns out that either $v_1 = v_2$, or $[v_1v_2] \subset \Sigma$.

If $v_1 = v_2$ then after deleting line segments $[v_1x_1]$ and $[v_2x_2]$ from $S$, application of Lemma~\ref{St_l} to line $Y_1Y_2$ gives that $S$ has at most two entering points.

If $[v_1v_2] \subset \Sigma$, then we consider two cases.
\begin{itemize}
\item The case where both points $x_1$ and $x_2$ have degree $1$; in this case $S$ is a full Steiner pseudo-network. The application~\ref{St_l} to the line $Y_1Y_2$ gives that $S$ contains at most 4 entering points, moreover if $S$ has 4 entering points then equality in lemma is achieved and by Remark~\ref{corollarystl} rays $[x_1y(x_1))$ and $[x_2y(x_2))$ have similar direction. 
Then the pass between $x_1$ and $x_2$ in $P$ has 3 branching points but $P$ has at most 2 vertices from $S$; which is a contradiction.

\item The case where at least one of points $x_1$ and $x_2$ has degree 2. Removing line segment $[v_1v_2]$ splits $\St(S)$ into two subnetworks $\St_1$ and $\St_2$. 
Suppose that one of them has at least two entering points (say, $\St_1$). Note that $\St_1 \setminus [x_1v_1]$ is a full pseudo-network, so by Lemma~\ref{St_l} it is a tripod; 
denote by $V$ the branching point of the tripod, and by $U_1$, $U_2$ the entering points in such a way that $M_r$ contains points $Y_1$, $U_1$, $U_2$, $Y_2$ in the mentioned order (some points may coincide). 
Then vector $\overrightarrow{U_1V}$ is directed away from line $Y_1Y_2$, hence vector $\overrightarrow{v_1v_2}$ also is directed away from line $Y_1Y_2$.

Summing up, no subtree can contain three entering points, and subtrees can not both contain two entering points simultaneously, which finishes the proof. 

\end{itemize}

\end{itemize}\qed

\section{Derivation in the picture}\label{diff}

Consider a point $y \in M$ such that $B_r(y) \cap \Sigma = \emptyset$. Suppose there exists an energetic point $x \in \partial B_r(y) \setminus M_r$. 
Our goal is to determine how the length of $\Sigma$ in the vicinity of point $x$ changes with the infinitesimal movement of $y$ along $M$ (and the corresponding movement of $x$). We are going to consider all possible options for the local structure of $\Sigma$ in the vicinity of $x$. Since radius of curvature of $M$ is greater than $r$, each $x$ corresponds to no more than two distinct $y$. Additionally, the degree of $x$ is either 1 or 2. Therefore, there are 4 cases to consider.

In all cases below we are going to move point $y$ along $M$ a distance $\varepsilon$ in such direction that the length of the arc covered by point $x$ increases (it changes a minimizer in a neighborhood of $x$).
The substitution of negative $\varepsilon$ corresponds to moving $y$ along $M$ in the opposite direction.

\paragraph{Case 1.} The degree of point $x$ is 1 (so $x$ is the end of some segment $[zx] \subset \Sigma$) and $y(x)$ is unique (look at the left half of Fig.~\ref{case12}). Points $z$, $x$, and $y(x)$ lie on one line by Lemma~\ref{stepen1}.
Let $|zx| = l$, let $\alpha$ be the angle between $(zy(x)]$ and $M$. We obtain the point $y(x_\varepsilon)$ by moving $y(x)$ a sufficiently small distance $\varepsilon$ along $M$; let $x_\varepsilon := [zy(x_\varepsilon)] \cap \partial B_r(y(x_\varepsilon))$. $M$ is smooth, so the distance between point $y(x_\varepsilon)$ and the tangent to $M$ at point $y(x)$ is $o(\varepsilon)$. 
By cosine rule for triangle $zy(x)y(x_\varepsilon)$,
\[
|zy(x_\varepsilon)| = \sqrt{|zy(x)|^2 + 2|zy(x)|\varepsilon \cos \alpha + \varepsilon^2} + o(\varepsilon) = |zy(x)| + \varepsilon \cos\alpha + o(\varepsilon).
\]
Therefore the derivative of the length of $\Sigma$ in the vicinity of $x$ with respect to the movement of $y(x)$ along $M$ in this case is $\cos\alpha$.

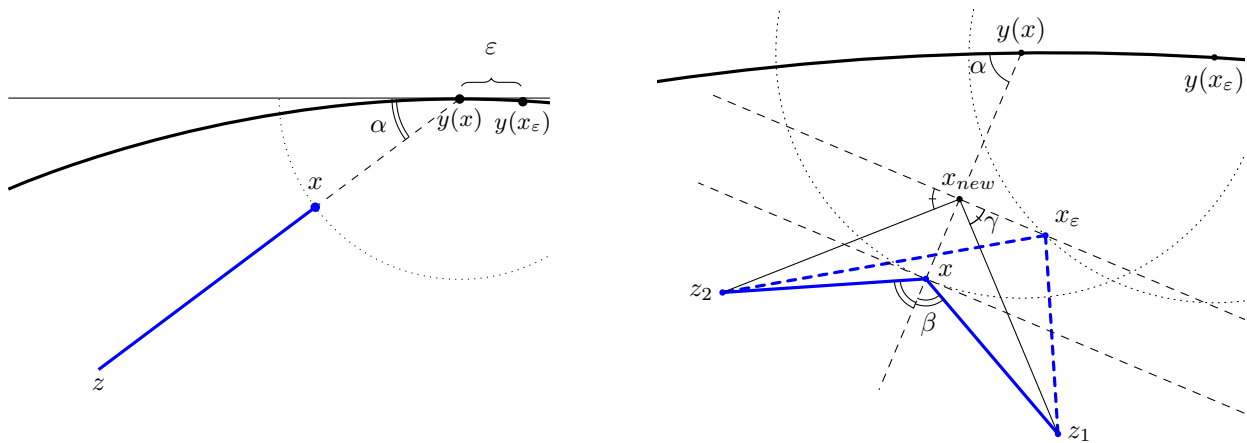
\begin{figure}[h]
    \centering
    \hfill
    \begin{tikzpicture}[scale=1.2]
    \def\eps{\varepsilon}
    
    \clip(-5,-3) rectangle (1,2);
    
    \draw[very thick] plot [smooth, tension=1] coordinates
        {(-5, 0) (0, 1) (5, 0)};
        
    \draw (-5, 1.01) -- (5, 1.01);
    \draw[dotted] (-2,1) arc (180:360:2);
    
    \draw[dashed] (0, 1) -- (-1.6, -.2);
    \draw[very thick, blue] (-4, -2) -- (-1.6, -.2);
    
    
    \draw (-.7, 1) arc (180 : 216.87 : 0.7);
    \draw (-.75, 1) arc (180 : 216.87 : 0.75);
    
    \node[left] at (-.7, 0.7) {$\alpha$}; 
    
        
    \draw [decorate , decoration={brace,mirror, amplitude=3pt},      
            xshift=0pt,yshift=4pt]
        (0.68, 1) -- (0.02, 1) node [midway,yshift=0.5cm] {$\eps$};
    
    \fill (0,1) circle (1.5pt)
        node[below] {{\small $y(x)$}};
        
    \fill[blue] (-1.6, -.2) circle (1.5pt);
    \node[above] at (-1.6, -.1) {$x$};
    
    \fill[blue] (-4, -2) circle (0pt) 
        node[below, black] {$z$};
        
    \fill (0.7, .97) circle (1.5pt)
        node[below] {{\small $y(x_\varepsilon)$}};

\end{tikzpicture}
\hfill
\begin{tikzpicture}[line cap=round,line join=round,>=triangle 45,x=1cm,y=1cm,scale=1.2]
\clip(-4.5, -.8) rectangle (2., 4.2);

\draw [very thick, shift={(-0.12513485800546506,-26.677521852562677)}]  plot[domain=1.3015317814364376:1.81024140245113,variable=\t]({1*30.392666753240377*cos(\t r)+0*30.392666753240377*sin(\t r)},{0*30.392666753240377*cos(\t r)+1*30.392666753240377*sin(\t r)});

\draw [dotted] (1.6662008271652229,3.6623085496751884) circle (2.7202892533708605cm);
\draw [dotted] (-0.4727808370220571,3.7131565644584974) circle (2.7202892533708605cm);

\draw [very thick, blue] (-3.783717507507988,1.056029328176248)-- (-1.530653918438071,1.2069885402389755);
\draw [very thick, blue] (-1.530653918438071,1.2069885402389755)-- (-0.06721948175656989,-0.5127334924261615);
\draw [very thick, blue, dashed] (-0.20613901407786517,1.6889061221660897)-- (-0.06721948175656989,-0.5127334924261615);
\draw [very thick, blue, dashed] (-0.20613901407786517,1.6889061221660897)-- (-3.783717507507988,1.056029328176248);
\draw [] (-1.1576901280692398,2.0905632733845843)-- (-0.06721948175656989,-0.5127334924261615);
\draw [] (-1.1576901280692398,2.0905632733845843)-- (-3.783717507507988,1.056029328176248);

\draw [shift={(-1.1576901280692398,2.0905632733845843)}] (-67.27213315390254:0.3003568321075131) arc (-67.27213315390254:-22.884987621529085:0.3003568321075131);
\draw[] (-0.9721088011312526,1.9044723323699622) -- (-0.9190855648632562,1.8513034920800702);

\draw [shift={(-1.530653918438071,1.2069885402389755)}] (-112.88498762152909:0.3003568321075131) arc (-112.88498762152909:-49.60316378273691:0.3003568321075131);

\draw [shift={(-1.530653918438071,1.2069885402389755)}] (-112.88498762152909:0.23527951848421863) arc (-112.88498762152909:-49.60316378273691:0.23527951848421863);

\draw [shift={(-1.530653918438071,1.2069885402389755)}] (-176.16681146032127:0.35041630412543195) arc (-176.16681146032127:-112.88498762152908:0.35041630412543195);

\draw [shift={(-1.530653918438071,1.2069885402389755)}] (-176.16681146032127:0.28533899050213746) arc (-176.16681146032127:-112.88498762152908:0.28533899050213746);

\draw [shift={(-1.1576901280692398,2.0905632733845843)}] (-67.27213315390254:0.3003568321075131) arc (-67.27213315390254:-22.884987621529085:0.3003568321075131);

\draw [shift={(-0.4727808370220571,3.7131565644584974)}] (-179.34460895105994:0.35041630412543195) arc (-179.34460895105994:-112.88498762152909:0.35041630412543195) ;

\draw [shift={(-1.1576901280692398,2.0905632733845843)}] (157.1150123784709:0.3003568321075131) arc (157.1150123784709:201.50215791084437:0.3003568321075131);

\draw [dashed] (-0.4727808370220571,3.7131565644584974)-- (-2.0483888830006287,-0.019558231771048984);
\draw [dashed] (-3.8848543096703176,3.2417205575080885)-- (2.4644358241962365,0.5616356414899564);
\draw [dashed] (2.3181285497377746,-0.41761257587217493)-- (-4.04973666644738,2.270313027354491);
\draw[] (-1.4204832205471947,2.093734673965504) -- (-1.4955669612551816,2.094640788417195);
\fill [blue] (-1.530653918438071,1.2069885402389755) circle (1pt);
\node[black, right] at (-1.5 ,1.3) {$x$};
\fill [blue] (-0.06721948175656989,-0.5127334924261615) circle (1pt)
    node[black, right] {$z_{1}$};
\fill [blue] (-3.783717507507988,1.056029328176248) circle (1pt)
    node[black, left] {$z_{2}$};
\fill [black] (-0.4727808370220571,3.7131565644584974) circle (1pt)
    node[above] {$y(x)$};
\fill [black] (1.6662008271652229,3.6623085496751884) circle (1pt)
    node[below] {$y(x_{\varepsilon})$};
\fill [blue] (-0.20613901407786517,1.6889061221660897) circle (1pt)
    node[black, above right] {$x_\varepsilon$};
\fill [black] (-1.1576901280692398,2.0905632733845843) circle (1pt);
\fill [white] (-1.1,2.25) circle (3pt);
\node[above] at (-1.05,2.1) {$x_{new}$};
\draw[] (-0.8,1.8) node {$\gamma$};
\draw[] (-0.95,3.5547777778793734) node {$\alpha$};
\draw[] (-1.5, 0.7) node {$\beta$};
\end{tikzpicture}
\hfill\phantom{}
    \caption{The first and second cases}
    \label{case12}
\end{figure}

\paragraph{Case 2.} The degree of point $x$ is 2 (so $x$ is the end of some segments $[z_1x], [xz_2] \subset \Sigma$) and $y(x)$ is unique (look at the right half of Fig.~\ref{case12}). The ray $[y(x)x)$ contains the bisector of the angle $z_1xz_2$ by Lemma~\ref{stepen2}. 
Let $|z_1x| = |z_2x| = l$ (we can shorten one of the segments if needed), $\beta = \frac12\angle z_1xz_2$; let the angle between the bisector of the angle $z_1xz_2$ and $M$ be $\alpha$; $y(x_\varepsilon)$ is obtained by moving $y(x)$ along $M$ a distance $\varepsilon$ (which is chosen to be sufficiently small after fixing $l$). Let $x_\varepsilon$ be such point on $\partial B_r(y(x_\varepsilon))$, that $[y(x_\varepsilon)x_\varepsilon)$ contains the bisector of the angle $z_1x_\varepsilon z_2$.

Consider the tangent to $B_r(y(x))$ at point $x$, and the parallel line going through point $x_\varepsilon$. Let point $x_{new}$ be the intersection between the last line and $[y(x)x)$. Note that 
$|z_1x_{new}| = |z_2x_{new}|$; denote this length as $l_{new}$. This equality also implies that $\angle z_1x_{new}x_\varepsilon+\angle z_2x_{new}x_\varepsilon=\pi$. 
We denote the angle $z_1x_{new}x_\varepsilon$ as $\gamma$ and write the cosine rule for triangles $z_1x_{new}x_\varepsilon$ and $z_2x_{new}x_\varepsilon$ using the fact that $|xx_{new}| = O(\eps)$:
\[
    |z_1x_\varepsilon| = 
    \sqrt{
        l_{new}^2 + |xx_{new}|^2 -
        2l_{new} |xx_{new}|\cos\gamma
    } = 
    l_{new} - |xx_{new}|\cos \gamma + o(\eps),
\]
\[
    |z_2x_\varepsilon| = 
    \sqrt{
        l_{new}^2 + |xx_{new}|^2 +
        2l_{new} |xx_{new}|\cos\gamma
    } = 
    l_{new} + |xx_{new}|\cos \gamma + o(\eps).
\]
Thus,
\begin{equation}\label{case2_1}
    |z_1x_\varepsilon| + |z_2x_\varepsilon| 
    - 2l_{new} = o(\varepsilon).
\end{equation}
 
Note that since $M$ is smooth,
\[
    |x_{new}x| = \varepsilon\cos\alpha + o(\varepsilon).
\]
Finally, we write the cosine rule for the triangle $z_1xx_{new}$:
\begin{equation}\label{case2_2}
    l_{new} = 
    \sqrt{
        l^2 + 
        (\varepsilon\cos\alpha + o(\varepsilon))^2 + 
        2 l(\varepsilon\cos\alpha + o(\varepsilon)) \cos\beta 
    } =
    l + \varepsilon\cos\alpha\cos\beta + o(\varepsilon).
\end{equation}
Combining~\eqref{case2_1} and~\eqref{case2_2}, we conclude that the derivative is
\[
   2\cos\alpha\cos\beta.
\]

\paragraph{In cases 3 and 4} $x$ corresponds to two points: $y_1(x)$ and $y_2(x)$. 
Let $y_1 (x) = y_1(x_\varepsilon)$, and let the point $y_2(x_\varepsilon)$ be obtained from $y_2(x)$ by moving it along $M$ a (possibly negative) distance $\varepsilon$. 
This uniquely determines the point $x_\varepsilon := \partial B_r(y_1(x)) \cap \partial B_r(y_2(x_\varepsilon))\cap N$. Let us find this point explicitly (look at the left half of Fig.~\ref{case34}).

The triangle $xy_1(x)y_2(x)$ is isosceles with two sides of length $r$; let $\angle xy_1(x)y_2(x) = \angle xy_2(x)y_1(x) =: \alpha$, $\angle x_\varepsilon y_1(x)y_2(x_\varepsilon ) = \angle x_\varepsilon y_2(x_\varepsilon )y_1(x) =: \alpha_\varepsilon$. 

We define the following coordinate system: the middle point of the segment $y_1(x)y_2(x)$ is the origin $O$; the $x$ axis is collinear to the ray
$[y_1(x)y_2(x))$; the $y$ axis is collinear to the ray $[Ox)$. Then
\[
O = (0,0), \quad x = (0, r\sin\alpha), \quad y_1(x) = (-r\cos\alpha,0), \quad y_2(x) = (r\cos\alpha,0).
\]
Denote the angle between $y_1(x)y_2(x)$ and $M$ as $\delta$. Then
\[
y_{2}(x_\varepsilon) = (r\cos\alpha + \varepsilon\cos\delta +o(\varepsilon), \varepsilon\sin\delta + o(\varepsilon)).
\]
Therefore, by the cosine rule for triangle $y_1(x)y_2(x)y_2(x_\varepsilon)$,
\[
|y_{1}(x)y_{2}(x_\varepsilon)| = \sqrt{(2r\cos\alpha + \varepsilon\cos\delta+o(\varepsilon))^2 + (\varepsilon\sin\delta+o(\varepsilon))^2} = 
2r\cos\alpha + \varepsilon \cos\delta  + o(\varepsilon).
\]
Let $O_\varepsilon$ be the middle point of segment $[y_{1}(x_\varepsilon)y_{2}(x_\varepsilon)]$. Then
\[
O_\varepsilon = \left(\frac{\varepsilon\cos\delta}{2}+o(\varepsilon),\frac{\varepsilon\sin\delta}{2}+o(\varepsilon)\right).
\]
By the definition of the cosine function,
\[
\alpha_\varepsilon = \arccos\left(\frac{|y_1(x)O_\varepsilon|}{r}\right) = \arccos\left(\cos\alpha + \frac{\varepsilon \cos\delta}{2r}  + o(\varepsilon) \right) = 
\alpha - \frac{\cos\delta}{2r\sin\alpha} \varepsilon + o(\varepsilon).
\]
Let $\Delta$ be the directed angle $\angle y_2(x)y_1(x)y_2(x_\varepsilon)$ (so $\Delta < 0$ when $\varepsilon$ is negative). 
By the sine rule for the triangle $y_2(x)y_1(x)y_2(x_\varepsilon)$,
\[
\frac{\varepsilon}{\sin \Delta} = \frac{|y_1(x)y_2(x)|}{\sin (\delta - \Delta + o(\varepsilon))} \geq |y_1(x)y_2(x)|, \quad \mbox {so} \quad \Delta = O(\varepsilon).
\]
Therefore,
\[
\Delta = \sin \Delta + o(\varepsilon) = \frac{ \varepsilon \sin (\delta + O(\varepsilon))}{|y_1(x)y_2(x)|}  = \frac{\varepsilon\sin\delta}{2r\cos\alpha} + o(\varepsilon).
\]
Writing out the sum of angles in the isosceles triangle $xy_1(x)x_\varepsilon$, we get
\[
\angle xy_1(x)x_\varepsilon = \alpha - \alpha_\varepsilon - \Delta = \left(\frac{\cos\delta}{2r\sin\alpha} - \frac{\sin\delta}{2r\cos\alpha}\right)\varepsilon + o(\varepsilon) = \frac{\cos(\alpha + \delta)}{r\sin(2\alpha)}\varepsilon + o(\varepsilon).
\] 
It follows that
\[
|xx_\varepsilon| = 2r\sin \frac{\angle xy_1(x)x_\varepsilon}{2} = \frac{\cos(\alpha + \delta)}{\sin(2\alpha)}\varepsilon + o(\varepsilon),
\]
and the angle between the segment $xx_\varepsilon$ and the $x$ axis (look at the right half of Fig.~\ref{case34}) is
\[
\pi - \alpha - \frac{\pi - \angle xy_1(x)x_\varepsilon}{2} = \frac{\pi}{2} - \alpha + \frac{\cos(\alpha + \delta)}{2r\sin(2\alpha)}\varepsilon + o(\varepsilon) = \frac{\pi}{2} - \alpha + o(1).
\]

\input{pictures/eng/case34}

\paragraph{Case 3.} The degree of point $x$ is 1 (so $x$ is the end of some segment $[zx] \subset \Sigma$) a there are two distinct points $y_1(x)$ and $y_2(x)$.   

Let $\beta$ be the angle between $[zx]$ and the $x$ axis (look at the right half of Fig.~\ref{case34}). Then
\[
\angle zxx_\varepsilon = \frac{3\pi}{2} - \alpha - \beta + o(1).
\]
By the cosine rule for the triangle $zxx_\varepsilon$,
\[
|zx_\varepsilon| = \sqrt{|zx|^2 - 2|xx_\varepsilon||zx|\cos\angle zxx_\varepsilon + |xx_\varepsilon|^2}  =  |zx| -|xx_\varepsilon|\cos\angle zxx_\varepsilon + o(\varepsilon) = 
|zx| + \frac{\cos(\alpha + \delta)\sin(\alpha+\beta)}{\sin(2\alpha)}\varepsilon + o(\varepsilon).
\]
So the derivative is equal to
\[
\frac{\cos(\alpha + \delta)\sin(\alpha+\beta)}{\sin(2\alpha)}.
\]

\paragraph{Case 4.} The degree of point $x$ is 2 (so it is the end of some segments $[z_1x], [xz_2] \subset \Sigma$) and there are two distinct points $y_1(x)$ and $y_2(x)$.   
Similar to the previous case, the derivative is equal to
\[
\frac{\cos(\alpha + \delta)}{\sin(2\alpha)} (\sin(\alpha+\beta) + \sin(\alpha+\gamma)),
\]
where $\beta$ and $\gamma$ are the angles between the $x$ axis and the segments $[z_1x]$ and $[z_2x]$, respectively.
 
\paragraph{Transitions between the cases.} Note that the second case can transform into the first case, and the other way around; similarly, the third case can turn into the fourth and vice versa.
The value of the derivative does not change in such transitions because
\[
2\cos \beta \cos \alpha = \cos \alpha \mbox{ when } \beta = \pi/3;
\]
\[
 \left(\sin(\alpha+\beta) + \sin(\alpha+\gamma)\right) \frac{\cos(\alpha + \delta)}{\sin(2\alpha)} = 
 2\sin \left(\frac{2\alpha+\beta+\gamma}{2}\right)\cos\left(\frac{\beta - \gamma}{2}\right)\frac{\cos(\alpha + \delta)}{\sin(2\alpha)} = \sin \left(\alpha+\beta+ \frac{\pi}{3}\right)\frac{\cos(\alpha + \delta)}{\sin(2\alpha)}
\]
when $\gamma - \beta = 2\pi/3$.
That means that even if the combinatorial structure of $\Sigma$ changes after moving $y$ along $M$ at the point $y_0$, the left and right derivatives at $y_0$ coincide.

\begin{proposition}
Let $x \in \Sigma$ be an energetic point, $y(x) \in M$ be an arbitrary corresponding point. 
Then the derivative of length of $\Sigma$ in a neighborhood of $x$ in the moving $y$ along $M$ is nonnegative.
\label{nonnegative}
\end{proposition}

\begin{proof}
Suppose the contrary. Then one may shift $y$ along $M$ and the length of $\Sigma$ will strictly decrease. Note that $\Sigma$ is still connected and $M$ is still covered by $\Sigma$; which is a contradiction.
\end{proof}

\begin{proposition}
Let $y \in M$ be a point such that $B_r(y) \cap \Sigma = \emptyset$ and $\partial{B_r(y)}$ contains energetic points $x_1$ and $x_2$.
Define $Y = \partial B_r(y) \cap M_r$. Then
\begin{itemize}
    \item [(i)] points $x_1$ and $x_2$ lie on the opposite sides of the line $(yY)$;
    \item [(ii)] derivative of length of $\Sigma$ in neighborhoods of $x_1$ and $x_2$ in the moving $y$ along $M$ are equal.
\end{itemize}
\label{diffproposition}
\end{proposition}

\begin{proof} 
Suppose the contrary to item~(i); without loss of generality, $\angle Yyx_1 > \angle Yyx_2$. Then
\[
B_r\left(B_\rho (x_1) \cap \Sigma \right) \cap M \subset B_r(x_2),
\]
where $\rho > 0$ is small enough. Thus $x_1$ is not energetic, which is a contradiction.

Now suppose the contrary to item~(ii). Without loss of generality, the derivative of the length of $\Sigma$ in a neighborhood of $x_1$ 
is bigger than the derivative in a neighborhood of $x_2$. Then after a shifting of $y$ along $M$ from $x_2$ to $x_1$ the length of $\Sigma$ strictly decreases. Note that $\Sigma$ is still connected and $M$ is still covered by $\Sigma$; this gives a contradiction.
\end{proof}

\section{Applications and open problems}
\label{rest}

\begin{itemize}
    \item Sometimes it is possible to ``derive in the picture'' in the case of a partially smooth $M$. 
    For this purpose one has to clarify the behavior of a considered competitor in a neighborhood of $B_r(y)$, with $y$ lying in the smooth part of $M$.

For instance we use an analog of Statement~\ref{diffproposition} during the pruning of cases in the proof of Theorem~\ref{rectangleT}. 
    
    \item Miranda, Paolini and Stepanov~\cite{miranda2006one} conjectured that all the minimizers for a circumference of radius $R > r$ are horseshoes. Theorem~\ref{horseshoeT} proves this conjecture with assumption $R > 4.98r$; for $4.98r \geq R > r$ the conjecture remains open.

    \item At the same time, the statement of Theorem~\ref{horseshoeT} for general $M$ needs an assumption on the minimal radius of curvature as we show below.
 
 Define a \textit{stadium} as the boundary of the $R$-neighborhood of a segment. By the definition, stadium has the minimal radius of curvature $R$. 
 If $R < 1.75r$ and a stadium is long enough, then there is a connected set $\Sigma'$ that has smaller length than an arbitrary horseshoe and covers $M$. 
 
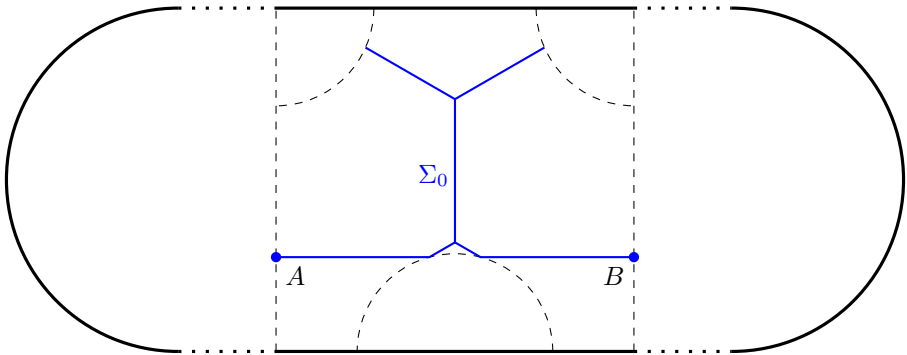
\begin{figure}[h]
    \centering

\begin{tikzpicture}[scale=1.3]

\foreach \y in {0, 3.51389977019889} {
    \draw[very thick, loosely dotted] 
        (-1, \y) -- (0, \y);
    \draw[very thick]
        (0, \y) -- (3.66, \y);
    \draw[very thick, loosely dotted] 
        (3.66, \y) -- (4.66, \y);
}
\draw[very thick] (-1, 0) arc (270:90:1.757);
\draw[very thick] (4.66, 0) arc (-90:90:1.757);

\draw[thin, dashed] (1.8298959646083486 + 1, 0.0) arc (0:180:1);

\draw[blue, thick] 
(0.0, 0.9652268503449002) -- 
(1.5684822553785844, 0.9652268503449002) --
(1.8298959646083486, 1.11615412573856) --
(2.0913096738381127, 0.9652268503449002) -- 
(3.659791929216697, 0.9652268503449002);

\node[blue] at (1.61,1.8) {$\Sigma_0$};

\draw[dashed] (0.0, 0.0) -- (0.0, 3.51389977019889);
\draw[dashed] (3.659791929216697, 0.0) -- (3.659791929216697, 3.51389977019889);

\fill[blue] (0.0, 0.9652268503449002) 
    circle (1.5pt) node[black, below right] {$A$};
\fill[blue] (3.659791929216697, 0.9652268503449002) 
    circle (1.5pt) node[black, below left] {$B$};

\draw[blue, thick] 
(1.8298959646083486, 1.11615412573856) -- 
(1.8298959646083486, 2.58208251632886);

\draw[thin, dashed] (1, 3.51389977019889) arc (0:-90:1);
\draw[thin, dashed] (2.66, 3.51389977019889) arc (180:270:1);

\draw[blue, thick]
(1.8298959646083486, 2.58208251632886) -- 
(0.9149479823041743, 3.11032798020668);

\draw[blue, thick] 
(1.8298959646083486, 2.58208251632886) -- 
(2.744843946912523, 3.11032798020668);

\end{tikzpicture}

    \caption{Horseshoe is not a minimizer for long enough stadium with $R < 1.75r$.}
    \label{stadion}
\end{figure}

Define $\Sigma_0$ as a locally Steiner tree depicted in Fig.~\ref{stadion}. 
Let $\Sigma'$ consist of copies of $\Sigma_0$, glued at points $A$ and $B$ along the length of the stadium. 
In the case $R < 1.75r$ the length of $\Sigma_0$ is strictly smaller than $2|AB|$. Thus for long enough stadium $\Sigma'$ has length $cL + O(1)$, where $L$ is the length of the stadium and $c < 2$ is a constant depend on $\Sigma_0$ and $R$. Obviously, any horseshoe has length $2L + O(1)$. 

This example leads to the following problems.

\begin{problem}
Find the minimal $c$ such that Theorem~\ref{horseshoeT} holds with the replacement of $5r$ with $cr$.
\end{problem}
 
\begin{problem}
Find the set of minimizers for a given stadium.
\end{problem}
 
\end{itemize}

\paragraph{Acknowledgments.} This work was supported by the Russian Science Foundation grant 16-11-10039.
The authors are grateful to Fedor Petrov for an analysis class.

\bibliographystyle{plain}
\bibliography{main}

\end{document}